\documentclass{amsart}
\usepackage{amssymb,latexsym}
\usepackage{amscd,amsthm}

\usepackage[all]{xy}
\usepackage{tikz}

\usepackage{tikz-cd}

\newtheorem{theorem}{Theorem}[section]

\newtheorem{lemma}[theorem]{Lemma}
\newtheorem{proposition}[theorem]{Proposition}
\newtheorem{corollary}[theorem]{Corollary}

\theoremstyle{definition}
\newtheorem{definition}[theorem]{Definition}

\newtheorem{remark}[theorem]{Remark}
\newtheorem*{notation}{Notation}
\newtheorem{example}[theorem]{Example}

\DeclareMathOperator{\Ext}{Ext}
\DeclareMathOperator{\Hom}{Hom}

\newcommand{\cat}[1]{\mathcal{#1}}           %% font for categories

\newcommand{\class}[1]{\mathcal{#1}}   %% font for classes

\newcommand{\Z}{\mathbb{Z}}
\newcommand{\Q}{\mathbb{Q/Z}}

\newcommand{\ch}{\textnormal{Ch}(R)}

\newcommand{\tilclass}[1]{\widetilde{\class{#1}}}

\newcommand{\dwclass}[1]{dw\widetilde{\class{#1}}}

\newcommand{\rightperp}[1]{#1^{\perp}}
\newcommand{\leftperp}[1]{{}^\perp #1}

\newcommand{\homcomplex}{\mathit{Hom}}

\begin{document}

\title[Duality pairs, generalized Gorenstein modules, and Ding injectives]{Duality pairs, generalized Gorenstein modules, and Ding injective envelopes}

\author{James Gillespie}
\address{J.G. \ Ramapo College of New Jersey \\
         School of Theoretical and Applied Science \\
         505 Ramapo Valley Road \\
         Mahwah, NJ 07430\\ U.S.A.}
\email[Jim Gillespie]{jgillesp@ramapo.edu}
\urladdr{http://pages.ramapo.edu/~jgillesp/}

\author{Alina Iacob}
\address{A.I. \ Department of Mathematical Sciences \\
         Georgia Southern University \\
         Statesboro (GA) 30460-8093 \\ U.S.A.}
\email[Alina Iacob]{aiacob@GeorgiaSouthern.edu}

\date{\today}



\begin{abstract}%% what do you actually prove?%%
Let $R$ be a general ring. Duality pairs of $R$-modules were introduced by Holm-J\o rgensen. Most examples satisfy further properties making them what we call semi-complete duality pairs in this paper. We attach a relative theory of Gorenstein homological algebra to any given semi-complete duality pair $\mathfrak{D} = (\class{L},\class{A})$. This generalizes the homological theory of the AC-Gorenstein modules defined by Bravo-Gillespie-Hovey, and we apply this to other semi-complete duality pairs. The main application is that the Ding injective modules are the right side of a complete (perfect) cotorsion pair, over any ring. Completeness of the Gorenstein flat cotorsion pair over any ring arises from the same duality pair. 
\end{abstract}

\maketitle

\section{introduction}\label{sec-intro}

Duality pairs were introduced by Holm-J\o rgensen in~\cite{holm-jorgensen-duality}, and complete duality pairs over commutative rings were defined in~\cite{gillespie-duality-pairs}.  In this paper, we extend this notion to noncommutative rings to show how a theory of relative Gorenstein homological algebra exists with respect to any given complete duality pair. In fact, this notion is too strong, and so we define \emph{semi-complete} duality pairs and develop the theory in this context. This will let us show that the Ding injective modules are the right side of a complete cotorsion pair over any ring $R$. 
As in~\cite{gillespie-Ding-Chen rings}, a module $N$ is said to be \emph{Ding injective} if $N = Z_0E$ for some exact complex of injectives $E$ such that $\Hom_R(A,E)$ remains exact for all FP-injective (absolutely pure) modules $A$.
Throughout, we let $R$ denote a ring with identity, and let $R^\circ := R^{\text{op}}$ denote its opposite ring.

The techniques go back to~\cite{bravo-gillespie-hovey} where the so-called \emph{level} and \emph{absolutely clean} modules played the central role in the AC-Gorenstein homological algebra that was developed there. In hindsight, the theory has good properties, enough to give both a projective and injective stable homotopy category on $R$-modules, simply because we have a (semi-)complete duality pair $(\class{L}, \class{A})$ where $\class{L}$ is the class of level $R$-modules and $\class{A}$ is the class of absolutely clean $R^\circ$-modules. Here, the central feature of being a duality pair is that a module $M$ is level (resp. absolutely clean) if and only if  $M^+ = \Hom_{\Z}(M,\Q)$ is absolutely clean (resp. level). One purpose of this paper is to give the definition of a \emph{semi-complete duality pair} for a general ring $R$ and to show that the arguments and theory of~\cite{bravo-gillespie-hovey} carry over to any semi-complete duality pair.  
This gives a unified theory encompassing everything in~\cite{bravo-gillespie-hovey, gillespie-duality-pairs, iacob-generalized-gorenstein}.
However, we also consider the semi-complete duality pair $\mathfrak{D} = (\langle Flat \rangle,\langle Inj \rangle)$ which is the (definable) duality pair generated by $R$. Here the theory is in agreement with two important results recently shown by 
Jan {\v{S}}aroch and Jan {\v{S}}\v{t}ov{\'{\i}}{\v{c}}ek in~\cite{saroch-stovicek-G-flat} --- The Gorenstein flat cotorsion pair, and the projectively coresolved Gorenstein flat cotorsion pair, are complete over any ring. See Corollary~\ref{cor-n-duality}(3). But what is new is that we get completeness of the Ding injective cotorsion pair this way, again over any ring. The Ding modules were introduced and studied by Nanqing Ding and coauthors and later named after Ding in~\cite{gillespie-Ding-Chen rings}.

In the process, we came across the following general theorem. We then obtain the results we want for duality pairs, and the various applications, as a corollary. To state the theorem, given a class of $R$-modules $\class{B}$, we say an $R$-module $N$ is \emph{Gorenstein $\class{B}$-injective} if
$N=Z_{0}E$ for some exact $\Hom(\class{B},-)$-acyclic complex of injective $R$-modules $E$. That is, both $E$ and $\Hom(B,E)$ are exact (acyclic) complexes for all $B\in\class{B}$. 
Those familiar with Gorenstein homological algebra will guess the definitions of the other concepts below, but see Definitions~\ref{Defs-relative-G-inj}, \ref{Defs-relative-G-pro}, \ref{Defs-relative-G-flat}, and~\ref{Defs-relative-G-flat-proj} for precise definitions.

\begin{theorem}\label{them-models}
Let $\class{B}$ be a class of $R^\circ$-modules containing all the injective modules.  Assume there exists a set (not just a class) $\class{S} \subseteq \class{B}$ such that each $B \in \class{B}$ is a transfinite extension of modules in $\class{S}$. 
\begin{enumerate}
\item There is a cofibrantly generated injective abelian model structure on $R^\circ$-Mod, the \textbf{Gorenstein $\class{B}$-injective model structure}, whose fibrant objects are the Gorenstein $\class{B}$-injective modules.
\item There is a cofibrantly generated projective abelian model structure on $R$-Mod, the \textbf{projectively coresolved Gorenstein $\class{B}$-flat model structure}, whose cofibrant objects are the projectively coresolved Gorenstein $\class{B}$-flat modules.
\item There is a cofibrantly generated  abelian model structure on $R$-Mod, the \textbf{Gorenstein $\class{B}$-flat model structure}, whose cofibrant objects (resp. trivially cofibrant objects) are the Gorenstein $\class{B}$-flat modules (resp. flat modules). This model structure shares the same class of trivial objects as the projective model structure.
\end{enumerate}
\end{theorem}

Each of these is Quillen equivalent to a model structure on chain complexes; See Theorem~\ref{thm-Gor-module} and Theorem~\ref{theorem-proj-coresolved-B-flat}.  For the injective case, it also follows that the Gorenstein $\class{B}$-injective modules are the right side of a perfect cotorsion pair.

Now if $\mathfrak{D} = (\class{L},\class{A})$ is a semi-complete duality pair, see Definition~\ref{def-complete duality pair}, then it follows from work of Holm-J\o rgensen that the class $\class{A}$ possesses a set $\class{S}$ as in Theorem~\ref{them-models}.
As a corollary, and by combining with~\cite[Theorem~A.6]{bravo-gillespie-hovey} for part (2), we get the following in Corollary~\ref{corollary-models}.

\begin{corollary}\label{corollary-models-intro}
The following abelian model structures are induced by any semi-complete duality pair $\mathfrak{D} = (\class{L},\class{A})$.
\begin{enumerate}
\item The \textbf{Gorenstein $\mathfrak{D}$-injective model structure} exists on $R^\circ$-Mod. It is a cofibrantly generated injective abelian model structure whose fibrant objects are the Gorenstein $\class{A}$-injective $R^\circ$-modules.
\item The \textbf{Gorenstein $\mathfrak{D}$-projective model structure} exists on $R$-Mod. It is a cofibrantly generated projective abelian model structure whose cofibrant objects are the Gorenstein $\class{L}$-projective $R$-modules, equivalently, the projectively coresolved Gorenstein $\class{A}$-flat $R$-modules. 
\item The \textbf{Gorenstein $\mathfrak{D}$-flat model structure}  exists on $R$-Mod. It is a cofibrantly generated abelian model structure whose cofibrant objects (resp. trivially cofibrant objects) are the Gorenstein $\class{A}$-flat modules (resp. flat modules). Moreover, the trivial objects in this model structure coincide with those in the Gorenstein $\mathfrak{D}$-projective model structure.
\end{enumerate}
\end{corollary}

Our main application, which stems from the semi-complete duality pair $\mathfrak{D} = (\langle Flat \rangle,\langle Inj \rangle)$, appears in Theorem~\ref{them-dings}. It proves that the Ding injective modules form an enveloping class over any ring $R$, and that they are the fibrant objects of a cofibrantly generated model structure on $R$-Mod. 

But in fact we are now able to obtain a relative homological algebra, for any ring $R$, and for each positive integer $1 \leq n \leq \infty$, from a (semi-)complete duality pair $\mathfrak{D}_n$. See Corollary~\ref{cor-n-duality}. This includes everything from the AC-Gorenstein homological algebra of~\cite{bravo-gillespie-hovey} ($n=\infty$), to the above Ding injectives and Saroch and Stovicek's (projectively coresolved) Gorenstein flats from~\cite{saroch-stovicek-G-flat} ($n=1$).

%other applications we have in mind are pointed out in Section~\ref{sec-example duality pairs}. Of course, applying this to the level-absolutely clean duality pair, we recover the AC-Gorenstein homological algebra. But we also have the Bravo-P\'erez duality pairs from~\cite[Cor.~3.7]{bravo-perez}.  Other interesting dualty pairs were first pointed out in~\cite{holm-jorgensen-duality}. However, our main application, leading to the result on Ding injectives, comes from the semi-complete duality pair $\mathfrak{D} = (\langle Flat \rangle,\langle Inj \rangle)$ which is also described in Section~\ref{sec-example duality pairs}. 

\

\emph{Conventions}: Throughout the paper $R$ denotes a ring with identity. Its opposite ring, $R^{\text{op}}$, will be denoted more succinctly by $R^\circ$. Recall that a left (resp. right) $R$-module is equivalent to a right (resp. left) $R^\circ$-module. Our convention throughout the entire paper is that the term \emph{$R$-module}, with the side left unspecified, may be fixed to mean either left or right $R$-module as the reader desires. But then one should realize that the term \emph{$R^\circ$-module} means a swap of sides with respect to that choice.  In other words, if we fix $R$-module to mean \emph{right} $R$-module, then ``$M$ is an $R^\circ$-module'' is just our way of saying $M$ is a \emph{left} $R$-module.

%%%%%%%%%%%%%%%%%%%%%%%%%%%%%%%%%%%%%%%%%%%%%%%%%%%%%%%%%%%%%%%%%%%%%
%%%%%%%%%%%%%%%%%%%%%%%%%%%%%%%%%%%%%%%%%%%%%%%%%%%%%%%%%%%%%%%%%%%%%

\section{Symmetric and semi-complete duality pairs}

Recall that for a given $R$-module $M$, its \emph{character module} is defined to be the $R^\circ$-module $M^+ = \Hom_{\Z}(M,\Q)$.

\begin{definition}\cite[Definition~2.1]{holm-jorgensen-duality}\label{def-duality pair}
A \emph{duality pair} over $R$ is a pair $(\class{M},\class{C})$, where $\class{M}$ is a class of $R$-modules and $\class{C}$ is a class of $R^\circ$-modules, satisfying the following conditions:
\begin{enumerate}
\item $M \in \class{M}$ if and only if $M^+ \in \class{C}$. 
\item $\class{C}$ is closed under direct summands and finite direct sums. 
\end{enumerate}
A duality pair $(\class{M},\class{C})$ is called \emph{perfect} if $\class{M}$ contains the module $R$, and is closed under coproducts and extensions. 
\end{definition}

The canonical example of a duality pair is when we take $\class{F}$ to be the class of all flat $R$-modules and $\class{I}$ to be the class of all injective $R^\circ$-modules. The following is the main result concerning (perfect) duality pairs.

\begin{theorem}\cite[Theorem~3.1]{holm-jorgensen-duality}\label{them-duality pair purity}
Let $(\class{M},\class{C})$ be a duality pair. Then the following hold: 
\begin{enumerate}
\item $\class{M}$ is closed under pure submodules, pure quotients, and pure extensions. 
\item If $(\class{M},\class{C})$ is perfect, then $(\class{M}, \rightperp{\class{M}})$ is a perfect cotorsion pair.  
\end{enumerate}
\end{theorem}

The following definition comes from~\cite{gillespie-duality-pairs} but it was only stated there for commutative rings. It combines Holm and J\o rgensen's above definition with a similar notion defined in~\cite[Appendix~A]{bravo-gillespie-hovey}.

\begin{definition}\label{def-symmetric duality pair}
By a \emph{symmetric duality pair} $\{\class{L}, \class{A}\}$ we mean: 
\begin{enumerate}
\item $\class{L}$ is a class of $R$-modules.
\item $\class{A}$ is a class of $R^\circ$-modules.
\item $(\class{L},\class{A})$ and $(\class{A},\class{L})$ are each duality pairs. 
\end{enumerate}
\end{definition}

An example of a symmetric duality pair is obtained by taking $\class{L}$ to be the class of all level $R$-modules and $\class{A}$ to be the class of all absolutely clean $R^\circ$-modules~\cite{bravo-gillespie-hovey}. 
Theorem~\ref{them-projectivecomplexes} below is a very useful result concerning symmetric duality pairs. It is a generalization of~\cite[Theorem~A.6]{bravo-gillespie-hovey} where it was proved for complexes of projectives. However, as suggested in~\cite[Remark~3.9]{estrada-gillespie-coherent-schemes}, the proof works for complexes of pure-projective $R$-modules because of Stovicek's work on chain complexes of pure-projectives. Recall that an $R$-module $M$ is \emph{pure-projective} if it is projective with respect to the class of all pure short exact sequences. This is the case  if and only if $M$ is a direct summand of a direct sum of finitely presented modules. In particular, projective modules and finitely presented modules are examples of pure-projective modules.

\begin{theorem}\label{them-projectivecomplexes}
Let $\{\class{L}, \class{A}\}$ be a symmetric duality pair with $R$-modules in $\class{L}$ and $R^\circ$-modules in $\class{A}$. 
\begin{enumerate}
\item Assume $P$ is a chain complex of pure-projective $R$-modules.  Then the tensor product of $P$ with any $R^\circ$-module $A \in \class{A}$ yields an exact complex if and only if $\Hom_R(P,L)$ is an exact complex for all $L \in \class{L}$. That is, $P$ is $\class{A}^{\otimes}$-acyclic if and only if it is $\Hom(-,\class{L})$-acyclic. 
\item Assume $Q$ a chain complex of pure-projective $R^\circ$-modules. Then the tensor product of $Q$ with any $R$-module $L \in \class{L}$ yields an exact complex if and only if $\Hom_{R^\circ}(Q,A)$ is an exact complex for all $A \in \class{A}$. That is, $Q$ is $\class{L}^{\otimes}$-acyclic if and only if it is $\Hom(-,\class{A})$-acyclic. 
\end{enumerate}  
\end{theorem}

\begin{proof}
Tensor products must be written on a particular side depending on the choice of $R$-module to mean \emph{left $R$-module} versus \emph{right $R$-module}.  So for definiteness, let us assume that $\class{L}$ is a class of \emph{left} $R$-modules and $\class{A}$ a class of \emph{right} $R$-modules. (Of course, versions of our argument still hold if we swap this choice.) So we are given a chain complex $P$ of pure-projective left $R$-modules and we wish to show that $A \otimes_R P$ is exact for all $A \in \class{A}$ if and only if $\Hom_R(P,L)$ for all $L \in \class{L}$.

\noindent $(\Longleftarrow)$ By adjoint associativity~\cite[Theorem~2.1.10]{enochs-jenda-book} we have
$$\Hom_{\Z}(A\otimes_{R}P, \Q) \cong \Hom_{R}(P,A^+).$$  So since $(\cat{A},\cat{L})$ is a duality pair it is easy to argue that if $\Hom_{R} (P,L)$ is
exact for all $L \in \cat{L}$, then $A\otimes_{R}P$ is exact for all
$A\in \cat{A}$. 

\noindent $(\Longrightarrow)$ Suppose $A\otimes_{R}P$ is exact for all
$A \in \cat{A}$. Then for any $L \in \cat{L}$, we see $L^{+}\otimes_{R}P$ is exact since $(\cat{L},\cat{A})$ is a duality pair. Using the above adjoint associativity again we conclude that $\Hom (P,
L^{++})$ is exact whenever $L \in \class{L}$. In other words, $\Hom (P,
K)$ is exact whenever $K \in \class{L}^{++}$ and we note $\class{L}^{++} \subseteq \class{L}$ since $\{\class{L}, \class{A}\}$ is a symmetric duality pair.

But for any $L$, the natural map $L\xrightarrow{} L^{++}$ is a pure
monomorphism~\cite[Proposition~5.3.9]{enochs-jenda-book}. So if $L \in
\cat{L}$, the quotient $L^{++}/L$ is also in $\cat{L}$ since $\cat{L}$
is closed under pure quotients by Theorem~\ref{them-duality pair purity}. 
We can therefore create a pure exact resolution
of $L\in \cat{L}$ by elements of $\cat{L}^{++}$. That is, we can
find a pure exact chain complex $X$ where $X_{i}=0$ for $i>0$,
$X_{0}=L$, and each of the $X_{i}$ for $i<0$ is in $\cat{L}^{++}$.
From this we can easily construct a short exact sequence 
\[
0 \xrightarrow{} S^{0}L \xrightarrow{} \widetilde{X} \xrightarrow{} Y
\xrightarrow{} 0,
\]
which we note is degreewise pure, has $Y$ as a pure exact complex (of modules in $\class{L}$), and has $\widetilde{X}$ bounded above with
entries in $\cat{L}^{++}$. 

Since $P$ has pure-projective components, applying $\homcomplex(P,-)$ yields another short exact sequence 
\[
0 \xrightarrow{} \homcomplex (P,S^{0}L) \xrightarrow{} \homcomplex (P,\widetilde{X}) \xrightarrow{} \homcomplex (P,Y)
\xrightarrow{} 0.
\]
By Stovicek's~\cite[Theorem~5.4]{stovicek-purity}, any chain map from a chain complex of pure-projectives to a pure exact complex must be null homotopic. In other words, $\homcomplex (P,Y)$ must be an exact complex. Moreover, $\homcomplex (P,S^{0}L) = \Hom_R(P,L)$, so to complete the proof it will suffice
to show that $\homcomplex (P,\widetilde{X})$ is exact. 
%This will follow from us showing that $\Ext^1_{\text{dw-pur}}(P,Z) = 0$ whenever $Z$ is a bounded above complex with components $Z_n \in \class{L}^{++}$. Here $\Ext^1_{\text{dw-pur}}$ denotes the Yoneda Ext group of degreewise pure extensions. But each such $Z$ equals an inverse transfinite (and degreewise split, so pure) extension of spheres $S^n(Z_n)$ on its components $Z_n \in \class{L}^{++}$. (This is dual to the fact that every bounded below complex is a transfinite extension of spheres on its components.) Now the result follows by the dual of the Eklof lemma, that is, since the kernel of the functor $\Ext^1_{\text{dw-pur}}(P,-)$ contains $\class{L}^{++}$ and is closed under inverse transfinite extensions. 
But if $Z$ is any \emph{bounded} complex with entries in $\cat{L}^{++}$, then we can prove $\homcomplex (P,Z)$ is exact by induction on the number of nonzero entries in $Z$. Now, like any bounded above complex, $\widetilde{X}$ is the inverse limit of its truncations $\widetilde{X}^{-n}$ for $n\in \Z$, where $(\widetilde{X}^{-n})_{i}=\widetilde{X}_{i}$ for $i\geq -n$ and is $0$ otherwise. This is a very simple inverse limit, in fact, it is an ``inverse transfinite extension'' (dual of transfinite extension) of the  spheres $S^i(\widetilde{X}_{i})$ on its components $\widetilde{X}_{i}$. One must check that $\homcomplex (P,\widetilde{X})=\varprojlim \homcomplex (P, \widetilde{X}^{-n})$ and that $\homcomplex (P,\widetilde{X})$ is an exact complex, completing the proof. 
\end{proof}

Referring to Definition~\ref{def-duality pair}, let us call $(\class{M},\class{C})$ a \textbf{semi-perfect} duality pair if it has all the properties required to be a perfect duality pair \emph{except} that $\class{M}$ may not be closed under extensions. 

\begin{definition}\label{def-complete duality pair}
By a \emph{semi-complete duality pair} $(\class{L},\class{A})$ we mean that $\{\class{L},\class{A}\}$ is a symmetric duality pair with $(\class{L},\class{A})$ being a semi-perfect duality pair. In this case, we call $\class{L}$ the \emph{projective class} and $\class{A}$ the \emph{injective class}. If $(\class{L},\class{A})$ is indeed perfect, then we call it a \emph{complete duality pair}. 
\end{definition}

\begin{remark}
If $(\class{L},\class{A})$ is a semi-complete duality pair then $\class{L}$ contains not just all projective $R$-modules, but also all flat $R$-modules by the argument in~\cite[Prop.~2.3]{gillespie-duality-pairs}. On the other hand, $\class{A}$ must contain all absolutely pure (i.e. FP-injective) $R^\circ$-modules. Indeed suppose $A$ is absolutely pure and embed it into an injective $I$. Note the monomorphism $A \hookrightarrow I$ is necessarily pure. The argument in~\cite[Prop.~2.3]{gillespie-duality-pairs} shows that $I \in \class{A}$. But since $(\class{A},\class{L})$ is also a duality pair we conclude from Theorem~\ref{them-duality pair purity}(1) that $A \in \class{A}$.
\end{remark}

\subsection{Examples of (semi-)complete duality pairs}\label{sec-example duality pairs}
Several classes of examples of duality pairs are given throughout~\cite{holm-jorgensen-duality, bravo-gillespie-hovey, bravo-perez}. We give a brief summary here of those that are complete duality pairs. We refer the reader to the original sources for more detailed references and unexplained terminology.

\begin{example}\label{example-level}
Let $R$ be any ring and let $\class{L}$ be the class of all level $R$-modules and $\class{A}$ the class of all absolutely clean $R^\circ$-modules~\cite{bravo-gillespie-hovey}. Then the \emph{level duality pair}, $(\class{L},\class{A})$, is a complete duality pair. Note then that a noncommutative ring $R$ admits \emph{two} level duality pairs - one where $\class{L}$ is the class of left $R$-modules and one where $\class{L}$ is the class of right $R$-modules.   
\end{example}

\begin{example}\label{example-BP}
Let $n$ be a natural number satisfying $2 \leq n \leq \infty$. In~\cite{bravo-perez}, Bravo and P\'erez give $n$-analogs to the level duality pairs. Here we let $\class{FP}_n\text{-Flat}$ denote their class of all $\text{FP}_n$-flat $R$-modules, and $\class{FP}_n\text{-Inj}$ their class of all $\text{FP}_n$-injective $R^\circ$-modules. It is shown in~\cite[Cor.~3.7]{bravo-perez} that we have a complete duality pair $(\class{FP}_n\text{-Flat},\class{FP}_n\text{-Inj})$. 
The class of $\text{FP}_n$-flat modules always sits between the usual class of flat modules ($n=1$) and the class of level modules ($n=\infty$), and the difference is only significant for  non-coherent rings. See~\cite{bravo-perez} for details.
\end{example}

\begin{example}
Many commutative rings $R$ have some interesting complete duality pairs attached to them. We refer the reader to the original source~\cite{holm-jorgensen-duality} and to the summary given in~\cite{gillespie-duality-pairs}. Depending on the hypotheses on the ring, there may be the \emph{Auslander-Bass duality pair} $(\class{A}^C_0, \class{B}^C_0)$, the \emph{$C$-Gorenstein flat dimension duality pairs} $(\class{GF}^C_n, \class{GI}^C_n)$ (where $C$ is a dualizing complex), or the \emph{depth-width duality pairs} $(\class{D}_n, \class{W}_n)$.
\end{example}

\begin{example}\label{ques-G-flat}
We see in~\cite[Remark~2.12]{estrada-iacob-perez-G-flat} that, given any ring $R$, it generates a semi-complete duality pair 
$(\langle R \rangle, \langle R^+ \rangle)$ where $\langle R \rangle$ is the \emph{definable class} (meaning it is closed under products, direct limits, and pure submodules) generated by $R$, and $\langle R^+ \rangle$ is the definable class generated by $R^+$. Moreover, they show $$\mathfrak{D} = (\langle R \rangle, \langle R^+ \rangle) = (\langle Flat\rangle, \langle Inj \rangle)$$ where  $\langle Flat\rangle$ is the definable class generated by the class of all flat $R$-modules and $\langle Inj \rangle$ is the definable class generated by the class of all injective $R^\circ$-modules.
Alternatively, using results from~\cite{prest-definable}, it is shown very succinctly in~\cite[Lemmas~5.5-5.7]{cortes-saroch} that $\mathfrak{D} = (\langle Flat\rangle, \langle Inj \rangle)$ is a semi-complete duality pair.  Moreover, $\langle Inj \rangle$ is precisely the class of all $R^\circ$-modules $M$ fitting into a short exact sequence
$$ 0\xrightarrow{} A \xrightarrow{} B \xrightarrow{} M \xrightarrow{} 0$$ where $A$ and $B$ are FP-injective (absolutely pure) $R^\circ$-modules.
\end{example} 

As in~\cite{gillespie-Ding-Chen rings}, a module $N$ is said to be \textbf{Ding injective} if $N = Z_0E$ for some exact complex of injectives $E$ such that $\Hom(A,E)$ remains exact for all FP-injective (absolutely pure) modules $A$.

As in~\cite{saroch-stovicek-G-flat}, a module $N$ is said to be \textbf{projectively coresolved Gorenstein flat} if $N = Z_0P$ for some exact complex of projectives $P$ which remains exact upon tensoring with any injective module $I$. So these are like the usual \emph{Gorenstein flat} modules we know from~\cite{enochs-jenda-book}, but defined via a complex of projectives, not just a complex of flats. 

We have the following results.
\begin{proposition}\label{prop-ding-thing}
Consider the semi-complete duality pair $\mathfrak{D} = (\langle Flat\rangle, \langle Inj \rangle)$ over any ring $R$.
\begin{enumerate}
\item An $R^\circ$-module $N = Z_0E$ is Ding injective if and only if it is Gorenstein $\langle Inj \rangle$-injective in the sense of Definition~\ref{Defs-relative-G-inj}. It just means that $\Hom(M,E)$ even remains exact for all $M \in \langle Inj \rangle$.
\item An $R$-module $N = Z_0F$ is Gorenstein flat if and only if it is Gorenstein $\langle Inj \rangle$-flat in the sense of Definition~\ref{Defs-relative-G-flat}. It means that the complex of flats $F$ even remains  exact upon tensoring it with any $M \in \langle Inj \rangle$. 
In particular, this is true for any projectively coresolved Gorenstein flat module $N = Z_0P$.
\end{enumerate}
\end{proposition}
\begin{proof}
Since $\langle Inj \rangle$ contains all FP-injective modules, any Gorenstein $\langle Inj \rangle$-injective is Ding injective.
On the other hand, suppose $N = Z_0E$ is Ding injective. We must show that $\Hom(M,E)$ remains exact for all $M \in \langle Inj \rangle$. But again, any such $M$ sits in a short exact sequence 
$$ 0\xrightarrow{} A \xrightarrow{} B \xrightarrow{} M \xrightarrow{} 0$$ where $A$ and $B$ are FP-injective (absolutely pure) $R^\circ$-modules. Applying the functor  $\Hom(-,E)$ yields, because each $E_n$ is injective, a short exact sequence of complexes 
$$ 0\xrightarrow{} \Hom(M,E) \xrightarrow{} \Hom(B,E) \xrightarrow{} \Hom(A,E) \xrightarrow{} 0.$$ 
Since $\Hom(B,E)$ and $\Hom(A,E)$ are both exact, it follows that $\Hom(M,E)$ is also exact. 

The fact for the Gorenstein flats (and projectively resolved) is proved similarly. But here one must first use~\cite[Lemma~5.3]{estrada-gillespie-coherent-schemes} (a fact first proved by Ding and Mao in~\cite[Lemma~2.8]{ding and mao 08})
 and that the short exact sequence containing $M \in \langle Inj \rangle$ is necessarily pure. 
\end{proof}

So now by Theorem~\ref{them-projectivecomplexes} (\cite[Theorem~A.6]{bravo-gillespie-hovey}) we have established the footnote in~\cite[Page~21]{saroch-stovicek-G-flat}. It includes a different proof of Saroch and Stovicek's~\cite[Theorem~4.4]{saroch-stovicek-G-flat}, 
that all projectively coresolved Gorenstein flat modules are Gorenstein projective. In fact, they are Ding projective in the sense of~\cite{gillespie-Ding-Chen rings}:

\begin{corollary}\cite[Theorem~4.4/Cor.~4.5]{saroch-stovicek-G-flat}\label{cor-ss} 
An $R$-module $N = Z_0P$ is projectively coresolved Gorenstein flat if and only if the complex $P$ in the definition satisfies that $\Hom_R(P,L)$ remains exact for all $L \in \langle Flat \rangle$.
\end{corollary}

\begin{remark}
We note that Corollary~\ref{cor-ss} was also proved by Estrada-Iacob-P\'erez in~\cite[Lemma~2.11/Remark~2.12]{estrada-iacob-perez-G-flat}; again by using that $(\langle Flat\rangle, \langle Inj \rangle)$ is a symmetric duality pair and applying~\cite[Appendix~A.6]{bravo-gillespie-hovey}.
\end{remark}
 
%%%%%%%%%%%%%%%%%%%%%%%%%%%%%%%%%%%%%%%%%%%%%%%%%%%%%%%%%%%%%%%%%%%%
%%%%%%%%%%%%%%%%%%%%%%%%%%%%%%%%%%%%%%%%%%%%%%%%%%%%%%%%%%%%%%%%%%%%

\section{Relative Gorenstein injective and projective modules}\label{sec-relative-G-inj}

Throughout this section, we let $\class{B}$ denote a class of $R$-modules and we assume $\class{B}$ contains all injective $R$-modules. 

We will prove a series of lemmas generalizing well-known results for the usual Gorenstein injectives. Their proofs depend only on the definition of a Gorenstein $\class{B}$-injective module, given below. 

\begin{definition}\label{Defs-relative-G-inj}
We will say that a chain complex $X$ of $R$-modules is \emph{$\Hom(\class{B},-)$-acyclic} if $\Hom(B,X)$ is an exact complex of abelian groups for all $B \in \class{B}$. If $X$ itself is also exact we will say that $X$ is an \emph{exact $\Hom(\class{B},-)$-acyclic} complex. 
We say an $R$-module $N$ is \emph{Gorenstein $\class{B}$-injective} if
$N=Z_{0}E$ for some exact $\Hom(\class{B},-)$-acyclic complex of injective $R$-modules $E$.
\end{definition}

\begin{notation}\label{notation-relative-G-inj}
We let $\class{GI}_{\class{B}}$ denote the class of all Gorenstein $\class{B}$-injective $R$-modules, and we set $\class{W} = \leftperp{\class{GI}_{\class{B}}}$.
\end{notation}

We note that $\class{W}$ is precisely the class of all modules $W$ such that $\Hom_R(W,E)$ remains exact for all exact $\Hom(\class{B},-)$-acyclic complexes of injectives $E$. Indeed it follows from the definition that $W \in \class{W}$ if and only if $\Ext^1_R(W,Z_{n}E)=0$ for all $n$ and all such $E$, and this is equivalent to $\Hom_R(W,E)$ being exact. In particular, $\class{B} \subseteq \class{W}$.

\begin{lemma}\label{lemma-characterize-G-Inj}
The following are equivalent.\\
(1) $N \in \mathcal{GI}_{\mathcal{B}}$\\
(2) There exists an exact and $\Hom(\mathcal{\class{B}}, -)$-acyclic complex\\
 $\ldots \rightarrow E_1 \rightarrow E_0 \rightarrow N \rightarrow 0$ with each $E_i$ injective, and $\Ext^i_R(B, N) =0$ for any $B \in \mathcal{B}$, for any $i \ge 1$. \\
(3) There is a short exact sequence $0 \xrightarrow{} N' \xrightarrow{} E \xrightarrow{} N \xrightarrow{} 0$ with $E$ injective and $N' \in \class{GI}_{\class{B}}$.
\end{lemma}

\begin{proof}
(1) implies (2) follows from the definition of Gorenstein $\mathcal{B}$-injective modules, since $\Ext^i_R(B, N) = H^{-i} \Hom(B, E) =0$, where $E$ is an exact and $
\Hom(\mathcal{B}, -)$ acyclic complex of injectives, such that $N = Z_0E$ .\\
(2) $\Rightarrow$ (1) Let $0 \rightarrow N \rightarrow E_{-1} \rightarrow E_{-2} \rightarrow \ldots$ be an injective resolution of $N$. Pasting it with the complex $\ldots \rightarrow E_1 \rightarrow E_0 \rightarrow N \rightarrow 0$  we obtain an exact complex of injectives $E$ such that $N = Z_0 E$. By hypothesis, $E$ remains exact when applying a functor $\Hom(B, -)$ with $B \in \mathcal{B}$.

(1) implies (3) is clear. For the converse, we imitate the argument from~\cite[Lemma~2.5]{Ding projective}. Briefly, note that $\Ext^i_R(B,N) = 0$ for all $B \in \class{B}$. Since $N' \in \class{GI}_{\class{B}}$, we may extend to the left to get a $\Hom(\class{B},-)$-acyclic resolution of injectives. Then we may paste this with any usual injective resolution of $N$. The resulting exact complex of injectives will be $\Hom(\class{B},-)$-acyclic because $\Ext^i_R(B,N) = 0$ for all $B \in \class{B}$.
\end{proof}

\begin{lemma}\label{lemma-injective-heart}
$\mathcal{W} \bigcap \mathcal{GI}_{\mathcal{B}}$ is the class of injective modules.
\end{lemma}

\begin{proof}
Let $G \in \mathcal{W} \bigcap \mathcal{GI}_{\mathcal{B}}$. By definition there is an exact sequence $$0 \rightarrow G' \rightarrow I \rightarrow G \rightarrow 0$$ with $G' \in \mathcal{GI}_{\mathcal{B}}$, and with $I$ an injective module. Since $G \in \mathcal{W}$, we have that $Ext^1_R(G,G')=0$. So the sequence is split exact, and therefore $G$ is injective.

On the other hand, $\class{W}$ contains every module in $\class{B}$. (This follows from the definition and computation of Ext by injective (co)resolutions.) Since $\class{B}$ contains all injective modules we conclude $\mathcal{W} \bigcap \mathcal{GI}_{\mathcal{B}}$ is exactly the class of all injective modules.
\end{proof}

\begin{lemma}\label{lemma-relative-G-Inj-summands}
The class $\mathcal{GI}_{\mathcal{B}}$ is closed under direct products and direct summands. 
\end{lemma}

\begin{proof}
It follows immediately from the definition that $\mathcal{GI}_{\mathcal{B}}$ is closed under direct sums. 

A direct argument we learned from Marco P\'erez will work in this context to prove $\mathcal{GI}_{\mathcal{B}}$ is closed under direct summands; see~\cite[Prop.~5.2 ]{bravo-gillespie-perez}.
\end{proof}

\begin{lemma}\label{lemma-relative-G-Inj-coresolving}
The class $\mathcal{GI}_{\mathcal{B}}$ is injectively coresolving. That is, it contains the injectives and for any short exact sequence
$0 \xrightarrow{} N' \xrightarrow{} N \xrightarrow{} N'' \xrightarrow{} 0$ with $N \in \class{GI}_{\class{B}}$, we have $N \in \class{GI}_{\class{B}}$ if and only if $N'' \in \class{GI}_{\class{B}}$.
\end{lemma}

\begin{proof}
The proof for closure under extensions follows just like the dual of the argument given in~\cite[Lemma 3.1]{enochs-iacob-jenda}.

Next assume $N' , N \in \mathcal{GI}_{\mathcal{B}}$. Write a short exact sequence $0 \xrightarrow{} N' \xrightarrow{} I \xrightarrow{} G \xrightarrow{} 0$ with $I$ injective and $G \in \class{GI}_{\class{B}}$ 
Construct the pushout diagram below:
$$\begin{CD}
 @. 0   @.   0   @. @.   \\
   @.        @VVV     @VVV      @.  @.\\
     0   @>>> N'    @>>>   N    @>>>   N''   @>>>    0 \\
    @.        @VVV    @VVV      @|   @.\\
     0   @>>> I @>>>  P    @>>>   N''    @>>>    0 \\
    @.        @VVV     @VVV      @.   @.\\
    @. G @=   G   @. @.   \\
    @.        @VVV     @VVV      @.  @.\\
    @. 0   @.   0   @. @.   \\
\end{CD}$$ 
The second row splits since $I$ is injective, forcing $N''$ to be a direct summand of $P$. But $P \in \class{GI}_{\class{B}}$ from the closure under extensions we just proved. Thus $N''  \in \class{GI}_{\class{B}}$, by Lemma~\ref{lemma-relative-G-Inj-summands}.
\end{proof}

\begin{remark}
Alternatively, one can prove the coresolving property and closure under direct summands by imitating the (dual of) the arguments in~\cite[Theorem~2.6]{Ding projective}, and citing~\cite[Prop.~1.4]{holm}. %Briefly, write a short exact sequence $0 \xrightarrow{} K \xrightarrow{} E \xrightarrow{} N' \xrightarrow{} 0$ with $E$ injective and $K \in \class{GI}_{\class{A}}$. Letting $P$ be the pullback of $K \xrightarrow{} N' \xleftarrow{} E$ , argue that $P \in K \in \class{GI}_{\class{A}}$ since this class has been shown to be closed under extensions. We get a short exact sequence $0 \xrightarrow{} P \xrightarrow{} E \xrightarrow{} N'' \xrightarrow{} 0$, which shows $N'' \in \class{GI}_{\class{A}}$ by Lemma~\ref{lemma-characterize-G-Inj}. 
\end{remark}

\begin{lemma}\label{lemma-relative-G-Inj-thick}
The class $\class{W}$ is thick, meaning it is closed under direct summands and satisfies the 2 out of 3 property on short exact sequences.
\end{lemma}

\begin{proof}
It is automatic that $\class{W}$ is closed under direct summands and extensions since it is defined as an Ext-orthogonal. 
In fact, by~\cite[Lemma~1.2.9]{garcia-rozas}, since $\class{GI}_{\class{B}}$ has been shown to be an injectively coresolving class, we may conclude that $\class{W} = \leftperp{\class{GI}_{\class{B}}}$ is a projectively resolving class, and, that $\Ext^i_R(W,N) = 0$ for all $W \in \class{W}$ and $N \in \class{GI}_{\class{B}}$ and $i \geq 1$.  

Now consider a short exact sequence $0 \xrightarrow{}   W'  \xrightarrow{} W \xrightarrow{}  W'' \xrightarrow{} 0$ with $W', W \in \class{W}$.
It is only left to show that $\Ext^1_R(W'',N) = 0$ for all $N \in \class{GI}_{\class{B}}$. We follow Holm's argument from~\cite[Lemma~3.5]{gillespie-recollement}.  First, for any such $N$, applying $\Hom(-,N)$ and looking at the resulting long exact sequence in Ext we get $\Ext^{\geqslant
    2}_R(W'',N)=0$. To see that $\Ext^1_R(W'',N)=0$ for every $N \in
  \class{GI}_{\class{B}}$, write a short exact sequence $0 \to N' \to E \to N \to
  0$, where $E$ is injective and $N' \in \class{GI}_{\class{B}}$. Applying
  $\Hom_R(W'',-)$ to this sequence gives $\Ext^1_R(W'',N) \cong
  \Ext^2_R(W'', N')$, which is zero by what we just proved.
\end{proof}

\begin{proposition}\label{proposition-relative-G-Inj}
Let $\class{B}$ be a class of modules containing the injectives. Suppose every module $M$ has a special $\class{GI}_{\class{B}}$-preenvelope. Then $(All, \class{W}, \class{GI}_{\class{B}})$ is an injective abelian model structure on $R$-Mod. In particular, $(\class{W}, \class{GI}_{\class{B}})$ is an hereditary cotorsion pair, and in fact, it is a perfect cotorsion pair.
\end{proposition}

\begin{proof}
To see that $(\class{W}, \class{GI}_{\class{B}})$ is a cotorsion pair (with enough injectives) we only need to show $\rightperp{\class{W}}\subseteq \class{GI}_{\class{B}} $. Given any $M \in \rightperp{\class{W}}$, write a special $\class{GI}_{\class{B}}$-preenvelope 
$0 \xrightarrow{}   M  \xrightarrow{} N \xrightarrow{}  W \xrightarrow{} 0$. So $W \in \class{W}$ and $N \in \class{GI}_{\class{B}}$. Since $M \in \rightperp{\class{W}}$, the sequence splits, making $M$ a direct summand of $N$. Therefore $M\in \class{GI}_{\class{B}}$ by Lemma~\ref{lemma-relative-G-Inj-summands}.

Since $(\class{W}, \class{GI}_{\class{B}})$ is a cotorsion pair with enough injectives, it also has enough projectives by the Salce trick~\cite[Prop.~7.1.7]{enochs-jenda-book}. Thus we have a complete cotorsion pair. 

By Lemma~\ref{lemma-relative-G-Inj-thick} the class $\class{W}$ is thick. So by~\cite[Prop.~3.1]{gillespie-ding-modules}, $\class{W}$ is closed under direct limits. Any complete cotorsion pair whose left side is closed under direct limits is a perfect cotorsion pair, by~\cite[Theorem~7.2.6]{enochs-jenda-book}. It is now clear too that $(All, \class{W}, \class{GI}_{\class{B}})$ is an injective abelian model structure, by Lemma~\ref{lemma-injective-heart}.
\end{proof}

In addition, using {\v{S}}aroch and {\v{S}}\v{t}ov{\'{\i}}{\v{c}}ek's~\cite[Theorem~5.6]{saroch-stovicek-G-flat}  we can see that, in any case, $(\class{W}, \class{GI}_{\class{B}})$ is at least always a cotorsion pair. We don't use the following result in this paper, but point it out for its own interest; it generalizes~\cite[Prop.~2]{iacob-generalized-gorenstein}. 

\begin{proposition}\label{prop-always-cot-pair}
Let $\class{B}$ be a class of modules containing the injectives. Then
$(\class{W}, \class{GI}_{\class{B}})$ is always an hereditary cotorsion pair with $\class{W}$ thick. 
\end{proposition}

\begin{proof}
For any class $\class{C}$, we have $\rightperp{\class{C}} = \rightperp{(\leftperp{(\rightperp{\class{C}})})}$, so we have a cotorsion pair $(\class{W},\rightperp{\class{W}})$, where $\class{W} = \leftperp{\class{GI}_{\class{B}}}$. We automatically have $\class{GI}_{\class{B}} \subseteq \rightperp{\class{W}}$, and we wish to show that $\rightperp{\class{W}}  \subseteq  \class{GI}_{\class{B}}$.
 
As pointed out the proof of Lemma~\ref{lemma-relative-G-Inj-thick}, $\class{W}$ is a projectively resolving class. Therefore, by~\cite[Lemma~1.2.8]{garcia-rozas}, $\rightperp{\class{W}}$ is an injectively coresolving class and $\Ext^i_R(W,N) = 0$ for all $W \in \class{W}$ and $N \in \rightperp{\class{W}}$ and $i \geq 1$. 
%So for any $M \in \rightperp{\class{W}}$, any injective coresolution $$ 0 \xrightarrow{}  M \xrightarrow{} E^0 \xrightarrow{} E^1 \xrightarrow{} E^2 \xrightarrow{} \cdots $$ must be $\Hom(\class{W},-)$-acyclic, so in particular must be $\Hom(\class{B},-)$-acyclic.
So by Lemma~\ref{lemma-characterize-G-Inj}, we only need to show that any $N \in \rightperp{\class{W}}$ admits a $Hom(\mathcal{\class{B}}, -)$-acyclic complex
$$\cdots \xrightarrow{} E_2 \rightarrow E_1 \rightarrow E_0 \rightarrow N \rightarrow 0$$ with each $E_i$ injective.
 But note that any $N \in \rightperp{\class{W}}$ must be Gorenstein injective, because $(\leftperp{\rightperp{\class{GI}_{\class{B}})}} \subseteq (\leftperp{\rightperp{\class{GI})}} = \class{GI}$, with the equality by~\cite[Theorem~5.6]{saroch-stovicek-G-flat}. So we have a short exact sequence \begin{equation}\label{equation-ses1}\tag{$*$} 0 \xrightarrow{} N_0 \xrightarrow{} E_0 \xrightarrow{} N \xrightarrow{} 0 \end{equation} with $E_0$ injective and $N_0$ Gorenstein injective. Let $W \in \class{W}$ be arbitrary, and we will show that $\Ext^1_R(W,N_0) = 0$. This will complete the proof, because repeating the argument ad infinitum produces the desired $Hom(\mathcal{\class{B}}, -)$-acyclic injective resolution. Write a short exact sequence 
\begin{equation}\label{equation-ses2}\tag{$**$} 0 \xrightarrow{}   W  \xrightarrow{} I \xrightarrow{}  W' \xrightarrow{} 0 \end{equation}
 with $I$ injective. Then $W' \in \class{W}$ by Lemma~\ref{lemma-relative-G-Inj-thick}.
 Applying $\Hom(W',-)$ to \eqref{equation-ses1} we get 
$$ 0 = \Ext^1_R(W',N) \xrightarrow{}   \Ext^2_R(W',N_0)  \xrightarrow{} \Ext^2_R(W',E_0) = 0$$
and so  $\Ext^2_R(W',N_0) =0$. On the other hand, applying $\Hom(-,N_0)$ to \eqref{equation-ses2} we get 
$$ 0 = \Ext^1_R(I,N_0) \xrightarrow{} \Ext^1_R(W,N_0)  \xrightarrow{} \Ext^2_R(W',N_0) = 0$$
and so $\Ext^1_R(W,N_0) =0$.
\end{proof}

\begin{proposition}\label{prop-injective model on complexes}
Let $\class{B}$ be any class of modules for which there exists a set (not just a class) $\class{S} \subseteq \class{B}$ such that each $B \in \class{B}$ is a transfinite extension of modules in $\class{S}$. 
Then there is a cofibrantly generated injective abelian model structure on the category of chain complexes whose fibrant objects are the exact $\Hom(\class{B},-)$-acyclic complexes of injectives.  We call this the \textbf{exact $\boldsymbol{\Hom(\class{B},-)}$-acyclic injective model structure}.
%Similarly, there is a model structure where the fibrant objects are the (not necessarily exact) $\Hom(\class{A},-)$-acyclic complexes of injectives. 
\end{proposition}

\begin{proof}
A detailed argument is given in~\cite[Lemma~3.3]{gillespie-duality-pairs} for commutative rings, but it certainly holds for noncommutative rings too. It shows this to be a consequence of~\cite[Theorem~4.1]{bravo-gillespie-hovey}. The point is that one can easily check that a complex $I$ of injective modules is exact and $\Hom(\class{B},-)$-acyclic if and only if $\Hom(R\oplus B,I)$ is exact, where $B$ is the single ``test module'' $B = \bigoplus_{N \in S} N$. 
\end{proof}

\begin{theorem}\label{thm-Gor-module}
Let $\class{B}$ be a class of modules containing the injectives. Assume there exists a set (not just a class) $\class{S} \subseteq \class{B}$ such that each $B \in \class{B}$ is a transfinite extension of modules in $\class{S}$. 
Then there is a cofibrantly generated injective abelian model structure on $R$-Mod, the \textbf{Gorenstein $\class{B}$-injective model structure},  whose fibrant objects are the Gorenstein $\class{B}$-injectives. In particular, $(\class{W}, \class{GI}_{\class{B}})$ is a complete hereditary cotorsion pair in $R$-Mod, cogenerated by a set. In fact, it is a perfect cotorsion pair.

The sphere functor $S^0(-) : R\text{-Mod} \xrightarrow{} \ch$ is a left Quillen equivalence from the Gorenstein $\class{B}$-injective model structure to the exact $\Hom(\class{B},-)$-acyclic injective model structure. 
\end{theorem}

\begin{proof}
We apply Proposition~\ref{proposition-relative-G-Inj}. For any object $M$, we can take a fibrant replacement of $S^0(M)$  in the exact $\Hom(\class{B},-)$-acyclic model structure. It is precisely a short exact sequence
\[
0 \xrightarrow{} S^{0}(M) \xrightarrow{} I \xrightarrow{}X \xrightarrow{} 0
\]
in which $I$ is an exact $\Hom(\class{B},-)$-acyclic complex of injectives and $X$ is
trivial in the exact $\Hom(\class{B},-)$-acyclic model structure.  By the snake lemma, we get a short exact sequence
\[
0 \xrightarrow{} M \xrightarrow{} Z_{0}I \xrightarrow{} Z_{0}X \xrightarrow{} 0.
\]
$Z_{0}I$ is Gorenstein $\class{B}$-injective by definition. By the argument in~\cite[Lemma~4.4]{gillespie-duality-pairs} we also get $Z_{0}X \in \class{W}$.

The functor $S^0(-) : R\text{-Mod} \xrightarrow{} \ch$ is left adjoint to the cycle functor $Z_0(-)$ and is a Quillen adjunction from the Gorenstein $\class{B}$-injective model structure to the exact $\Hom(\class{B},-)$-acyclic injective model structure. The argument from~\cite[Theorem~5.8]{bravo-gillespie-hovey} generalizes to show that it is indeed a Quillen equivalence
\end{proof}

\begin{corollary}\label{cor-well-generated}
The full subcategory $\class{GI}_{\class{B}} \subseteq R\textnormal{-Mod}$ is a Frobenius category whose projective-injective objects are precisely the usual injective $R$-modules.  
The canonical functor $\gamma : R\textnormal{-Mod} \xrightarrow{} \textnormal{Ho}(R\textnormal{-Mod})$ takes all projective modules and all modules in $\class{B}$ to 0, and we have a triangulated equivalence to the stable category 
$$\textnormal{Ho}(R\textnormal{-Mod}) \cong \textnormal{St}(\class{GI}_{\class{B}}).$$  
Moreover, these are well-generated triangulated categories. 
\end{corollary}

\begin{proof}
The canonical functor $\gamma$ takes precisely $\class{W}$ to 0, and $\class{W}$ contains $\class{B}$, and certainly all projectives. The cotorsion pair $(\class{W}, \class{GI}_{\class{B}})$ is hereditary in the sense that $\class{W}$ is closed under taking cokernels of monomorphisms. Thus the Frobenius equivalence follows from a general result about hereditary abelian model structures~\cite[Theorem~4.3]{gillespie-hereditary-abelian-models}. 
We also point out that the homotopy category is a well generated category in the sense of~\cite{neeman-well generated}. Indeed once we have a cofibrantly generated model structure on a locally presentable (pointed) category, a main result from~\cite{rosicky-brown representability combinatorial model srucs} is that its homotopy category is well generated.
\end{proof}

\subsection{Gorenstein $\class{B}$-projectives modules}

Much of what we have done above has a projective dual. To describe, let $\class{B}$ denote a class of modules, but now assume it contains all of the projective modules (instead of the injective modules). 

\begin{definition}\label{Defs-relative-G-pro}
We say an $R$-module $M$ is \emph{Gorenstein $\class{B}$-projective} if
$M=Z_{0}Q$ for some exact and $\Hom(-, \class{B})$-acyclic complex of projective $R$-modules $Q$.
\end{definition}

\begin{notation}\label{notation-relative-G-proj}
We let $\class{GP}_{\class{B}}$ denote the class of all Gorenstein $\class{B}$-projective $R$-modules, and we set $\class{V} = \rightperp{\class{GP}_{\class{B}}}$.
\end{notation}

We leave it to the reader to formulate and verify the duals of the sequence of Lemmas~\ref{lemma-characterize-G-Inj}--\ref{lemma-relative-G-Inj-thick}.
We get the following result, dual to Proposition~\ref{proposition-relative-G-Inj}. But note that we don't get  a \emph{perfect} cotorsion pair. For the Gorenstein $\class{B}$-injectives, that conclusion relies on~\cite[Prop.~3.1]{gillespie-ding-modules} and~\cite[Theorem~7.2.6]{enochs-jenda-book}; we don't have duals for those.

\begin{proposition}[Dual of Proposition~\ref{proposition-relative-G-Inj}]\label{proposition-relative-G-Proj}
Let $\class{B}$ be a class of modules containing the projectives.
Suppose every module $M$ has a special $\class{GP}_{\class{B}}$-precover. Then $(\class{GP}_{\class{B}}, \class{V})$ is a complete hereditary cotorsion pair. In fact, $(\class{GP}_{\class{B}}, \class{V}, All)$ is a projective abelian model structure on $R$-Mod.
\end{proposition}

There is however a dual for Proposition~\ref{prop-always-cot-pair}. Note that the proof of Proposition~\ref{prop-always-cot-pair} only uses that the Gorenstein injectives are the right side of a cotorsion pair (not completeness). It was just shown in~\cite[Cor.~3.4]{cortes-saroch} that the Gorenstein projectives are the left half of a cotorsion pair, and this will give us the dual of Proposition~\ref{prop-always-cot-pair}. However, this is shown directly in~\cite[Them.~3.3]{cortes-saroch}!  

So this is as far as we know how to go by working straight from the definition of the Gorenstein $\class{B}$-projectives.  However, IF we can build the projective model structure on $\ch$ that is dual to the one in Proposition~\ref{prop-injective model on complexes}, then the dual of Theorem~\ref{thm-Gor-module} and its Corollary~\ref{cor-well-generated} will hold by duality arguments.  We make a precise statement for later use. 

\begin{theorem}[Dual of Theorem~\ref{thm-Gor-module}]\label{theorem-projectivemodels}
Let $\class{B}$ be a class of modules containing the projectives.
Suppose we have constructed a projective abelian model structure on the category of chain complexes whose cofibrant objects are the exact $\Hom(-,\class{B})$-acyclic complexes of projectives.  Call this the \textbf{exact $\boldsymbol{\Hom(-,\class{B})}$-acyclic projective model structure}. Then there is a projective abelian model structure on $R$-Mod, the \textbf{Gorenstein $\class{B}$-projective model structure}, in which the cofibrant objects are the Gorenstein $\class{B}$-projectives. In particular, $(\class{GP}_{\class{B}}, \class{V})$ is a complete hereditary cotorsion pair in $R$-Mod.

In this case, the sphere functor $S^0(-) : R\text{-Mod} \xrightarrow{} \ch$ is a right Quillen equivalence from the Gorenstein $\class{B}$-projective model structure to the exact $\Hom(-,\class{B})$-acyclic projective model structure. 
\end{theorem}

\begin{proof}
Let us just comment on how the proof of Theorem~\ref{thm-Gor-module} dualizes. 
A main point is that the functor $S^0(-) : R\text{-Mod} \xrightarrow{} \ch$ is also right adjoint, to the functor $X \mapsto X_0/B_0X$. 
The idea is to apply Proposition~\ref{proposition-relative-G-Proj}. So for any object $M$, we take a short exact sequence
\[
0 \xrightarrow{} X \xrightarrow{} P \xrightarrow{}S^0(M) \xrightarrow{} 0
\]
where $P$ is an exact $\Hom(-,\class{B})$-acyclic complex of projectives and $X$ is
trivial in the exact $\Hom(-,\class{B})$-acyclic model structure.  By the snake lemma, we get a short exact sequence
\[
0 \xrightarrow{} X_0/B_0X \xrightarrow{} P_0/B_0P \xrightarrow{} M \xrightarrow{} 0.
\]
$P_0/B_0P \cong Z_{-1}P$ is Gorenstein $\class{B}$-projective by definition. The argument of~\cite[Lemma~4.4]{gillespie-duality-pairs} dualizes, and we get $X_{0}/B_0X \in \class{V}$. So Proposition~\ref{proposition-relative-G-Proj} applies.

Again, the functor $S^0(-) : R\text{-Mod} \xrightarrow{} \ch$ is right adjoint to the functor $X \mapsto X_0/B_0X$. The argument from~\cite[Theorem~8.8]{bravo-gillespie-hovey} generalizes to show that they form a Quillen equivalence from the exact $\Hom(-,\class{B})$-acyclic projective model structure to the Gorenstein $\class{B}$-projective model structure. 
\end{proof}

\begin{remark}
In the above scenario of Theorem~\ref{theorem-projectivemodels}, the dual of Corollary~\ref{cor-well-generated} also holds. However, the conclusion that the homotopy category is well-generated is dependent on showing the model structure to be cofibrantly generated. 
\end{remark}

\section{Relative Gorenstein flat and projectively coresolved modules}\label{sec-coresolved}

We again let $\class{B}$ denote a class of modules containing all injective modules. However, we now assume that all the modules in $\class{B}$ are $R^\circ$-modules, where $R^\circ$ denotes the oppose ring $R^{\text{op}}$. The following notion of Gorenstein $\class{B}$-flat module was studied in~\cite{estrada-iacob-perez-G-flat}.

\begin{definition}\label{Defs-relative-G-flat}
We will say that a chain complex $X$ of $R$-modules is \emph{$\class{B}^{\otimes}$-acyclic} if the tensor product of $X$ with any $B \in \class{B}$ yields an exact complex of abelian groups. If $X$ itself is also exact we will say that $X$ is an \emph{exact $\class{B}^{\otimes}$-acyclic} complex. 
We say an $R$-module $N$ is \emph{Gorenstein $\class{B}$-flat} if
$N=Z_{0}F$ for some exact $\class{B}^{\otimes}$-acyclic complex of flat $R$-modules $F$. 
\end{definition}

\begin{notation}\label{notation-relative-G-flat}
We let $\class{GF}_{\class{B}}$ denote the class of all Gorenstein $\class{B}$-flat $R$-modules. We set $\class{GC}_{\class{B}} = \rightperp{\class{GF}_{\class{B}}}$ and call this  the class of all \emph{Gorenstein $\class{B}$-cotorsion} modules.\\
\end{notation}

Estrada-Iacob-P\'erez show that $\class{GF}_{\class{B}}$ is a Kaplansky class and closed under direct limits, and that gives us the following result.

\begin{proposition}\cite[Corollary~2.20]{estrada-iacob-perez-G-flat}\label{G-B-flat}
Suppose the class $\class{GF}_{\class{B}}$ is closed under extensions. Then $(\class{GF}_{\class{B}}, \class{GC}_{\class{B}})$ is a perfect hereditary cotorsion pair, cogenerated by a set. 
\end{proposition}

Now let $(\class{F},\class{C})$ denote Enochs' flat cotorsion pair. Here $\class{F}$ denotes the class of all flat $R$-modules and $\class{C}$ the class of all cotorsion $R$-modules. It is then shown in~\cite[Proposition~3.1]{estrada-iacob-perez-G-flat} that $\class{GF}_{\class{B}}\cap\class{GC}_{\class{B}} = \class{F}\cap \class{C}$, as long as $\class{GF}_{\class{B}}$ is closed under extensions. Applying~\cite[Theorem~1.2]{gillespie-hovey triples}, it proves the following.\\

\begin{theorem}~\cite[Theorem~3.2]{estrada-iacob-perez-G-flat}\label{thm-Gor-flat-mod}
Let $\class{B}$ be a class of $R^\circ$-modules containing the injectives.
Assume that the Gorenstein $\class{B}$-flat modules are closed under extensions. Then there is cofibrantly generated abelian model structure on $R$-Mod, the \textbf{Gorenstein $\class{B}$-flat model structure}, corresponding to the cotorsion pairs $(\class{GF}_{\class{B}}, \class{GC}_{\class{B}})$ and $(\class{F}, \class{C})$.
\end{theorem}

We will see below in Proposition~\ref{prop-proj-coresolved-B-flat} that, as in~\cite[Theorem~4.11]{saroch-stovicek-G-flat} and~\cite[Theorem~2.14]{estrada-iacob-perez-G-flat}, closure under extensions comes free for the classes $\class{B}$ we will consider in this paper. In particular, this is the case whenever $\class{B}$ is the injective class for some semi-complete duality pair. More generally, when $\class{B}$ satisfies the hypotheses of Theorem~\ref{theorem-proj-coresolved-B-flat}.

%\begin{proof} We know that Enochs' flat cotorsion pair, $(\class{F},\class{C})$, is cogenerated by a set. Under the assumption that $\class{GF}_{\mathfrak{D}}$ is closed under extensions we obtain from~\cite[Corollary~2.7]{estrada-iacob-perez-G-flat} that $(\class{GF}_{\mathfrak{D}},\class{GC}_{\mathfrak{D}})$ is also a cotorsion pair cogenerated by a set. The rest of the theorem follows by imitating the proofs of Theorem~5.1, Proposition~5.2 and Proposition~5.5 of~\cite{estrada-gillespie-coherent-schemes}.\end{proof}

\subsection{Projectively coresolved Gorenstein $\class{B}$-flat modules}

$\class{B}$ still denotes a class of $R^\circ$-modules containing all injectives.
The following relative version of  {\v{S}}aroch and {\v{S}}\v{t}ov{\'{\i}}{\v{c}}ek's projectively coresolved Gorenstein flat modules was studied in~\cite{estrada-iacob-perez-G-flat}.

\begin{definition}\label{Defs-relative-G-flat-proj}
We say an $R$-module $N$ is \emph{projectively coresolved Gorenstein $\class{B}$-flat} if
$N=Z_{0}Q$  for some exact $\class{B}^{\otimes}$-acyclic complex of projective $R$-modules $Q$. 
\end{definition}

\begin{notation}\label{notation-relative-G-flat-proj}
We let $\class{PGF}_{\class{B}}$ denote the class of all projectively coresolved Gorenstein $\class{B}$-flat $R$-modules, and we set $\class{V} = \rightperp{\class{PGF}_{\class{B}}}$.
\end{notation}

\begin{lemma}\label{lemma-proj-perp}
The class $\class{V} := \rightperp{\class{PGF}_{\class{B}}}$ equals the class of all $R$-modules $V$ such that $\Hom_R(Q,V)$ is acyclic for every exact and $\class{B}^{\otimes}$-acyclic complex of projectives $Q$. Equivalently, $\Ext^1_{\ch}(Q,S^0V) = 0$ for all such $Q$. 
\end{lemma}

\begin{proof}
Note that the class of all exact $\class{B}^{\otimes}$-acyclic complexes of projectives $Q$, is closed under suspensions. So we have that $V \in \rightperp{\class{PGF}_{\class{B}}}$ if and only if we have $\Ext^1_R(Z_nQ,V) = 0$ for all such $Q$. Since $Q$ is an exact complex of projectives,  this happens if and only if $\Hom_R(Q,V)$ is exact for all such $Q$. But $\Hom_R(Q,V)=\homcomplex(Q,S^0V)$, and since $Q$ is a complex of projectives this complex is exact if and only if $\Ext^1_{\ch}(Q,S^0V) = 0$ for all such $Q$.
\end{proof}

Next we have an analog of Proposition~\ref{proposition-relative-G-Inj} for the class of $\mathcal{PGF}_{\mathcal{B}}$ modules:

\begin{proposition}[Analog of Proposition~\ref{proposition-relative-G-Inj} ]\label{prop-proj-coresolved-B-flat}
Let $\class{B}$ be a class of $R^\circ$-modules containing the injectives.
Suppose every module $M$ has a special $\mathcal{PGF}_{\mathcal{B}}$-precover. Then $(\mathcal{PGF}_{\mathcal{B}}, \mathcal{V})$ is a complete hereditary cotorsion pair. In fact, $(\class{PGF}_{\class{B}}, \class{V}, All)$ is a projective abelian model structure on $R$-Mod. 

Moreover,  Gorenstein $\class{B}$-flat modules are closed under extensions and $(\class{GF}_{\class{B}}, \class{V}, \class{C})$ is a cofibrantly generated abelian model structure on $R$-Mod. That is, the \textbf{Gorenstein $\class{B}$-flat model structure} of Theorem~\ref{thm-Gor-flat-mod} exists and shares the same class of trivial objects as the projective model structure.  
\end{proposition}

\begin{proof}
We show that $^\bot \class{V} \subseteq \mathcal{PGF}_{\mathcal{B}}$, and therefore $(\mathcal{PGF}_{\mathcal{B}}, \class{V})$ is a cotorsion pair. Let $M \in \leftperp{\class{V}}$. Consider an exact sequence $0 \rightarrow A \rightarrow D \rightarrow M \rightarrow 0$ with $D \in \mathcal{PGF}_{\mathcal{B}}$ and $A \in \class{V} = \mathcal{PGF}_{\mathcal{B}}^\bot$. Since $\Ext^1_R(M,A) =0$ we have $D \cong A \oplus M$, so $M \in \mathcal{PGF}_{\mathcal{B}}$. Indeed $\mathcal{PGF}_{\mathcal{B}}$ is closed under direct summands for the following reason.  It is shown in~\cite[Theorem~2.10]{estrada-iacob-perez-G-flat} that $\mathcal{PGF}_{\mathcal{B}}$ is a resolving class, (as long as $\class{B}$ contains all the injective modules). It is clearly closed under direct sums as well. Therefore, $\mathcal{PGF}_{\mathcal{B}}$ is closed under direct summands by~\cite[Prop.~1.4]{holm}. Therefore, $(\mathcal{PGF}_{\mathcal{B}}, \mathcal{V})$ is a cotorsion pair with enough projectives. Therefore, it also has enough injectives by the Salce trick~\cite[Prop.~7.1.7]{enochs-jenda-book}. So $(\mathcal{PGF}_{\mathcal{B}}, \mathcal{V})$ is a complete cotorsion pair. 

The pair is hereditary: if $N \in  \mathcal{PGF}_{\mathcal{B}}$ then, by definition, there is an exact sequence $0 \rightarrow N' \rightarrow P \rightarrow N \rightarrow 0$ with $P$ projective and $N' \in  \mathcal{PGF}_{\mathcal{B}}$. Then for any $V \in \class{V}$, the exact sequence $0 = \Ext^1_R(N',V) \rightarrow \Ext^2_R(N,V) \rightarrow \Ext^2_R(P,V)=0$ gives that $\Ext^2_R(N,V)=0$. Similarly, $\Ext^i_R(N,V)=0$ for all $i \ge 1$, and all $V \in \class{V}$.

Any right orthogonal class, in particular $\mathcal{V}$, is closed under direct summands. The fact that the class $\mathcal{V}$ has the 2 out of 3 property on short exact sequences follows from Lemma~\ref{lemma-proj-perp}: For every exact and $\class{B}^{\otimes}$-acyclic complex of projectives $Q$, apply the functor $Hom_R(Q,-)$ to any short exact sequence of $R$-modules. The 2 out of 3 property for exactness of cochain complexes gives the result.

Since we have $(\mathcal{PGF}_{\mathcal{B}}, \mathcal{V})$ is a complete cotorsion pair and $\mathcal{V}$ is thick, we will get the projective abelian model structure  $(\class{PGF}_{\class{B}}, \class{V}, All)$ by applying~\cite[Proposition~3.4]{bravo-gillespie-hovey}, once we see that $\mathcal{V}$ contains all projective modules. But every module in $\mathcal{PGF}_{\mathcal{B}}$ is a projectively coresolved Gorenstein flat module in the sense of  {\v{S}}aroch and {\v{S}}\v{t}ov{\'{\i}}{\v{c}}ek~\cite{saroch-stovicek-G-flat}, because we are assuming $\class{B}$ contains all injectives. A key result they show in~\cite[Theorem~4.4]{saroch-stovicek-G-flat} (see Corollary~\ref{cor-ss}) is that every such module is Gorenstein projective. It follows that $\class{V}$ contains all projective modules. So $(\class{PGF}_{\class{B}}, \class{V}, All)$ is a projective abelian model structure on $R$-Mod. 

In fact, it follows from {\v{S}}aroch and {\v{S}}\v{t}ov{\'{\i}}{\v{c}}ek's~\cite[Theorem~4.4]{saroch-stovicek-G-flat} (see Corollary~\ref{cor-ss}) that $\class{V}$ contains all flat modules. Therefore, the claim that $\class{V}$ is also the class of trivial objects in the Gorenstein $\class{B}$-flat model structure will follow immediately from~\cite[Proposition~3.2]{gillespie-recollement} combined with~\cite[Lemma~2.3(1)]{gillespie-models-for-hocats-of-injectives}, once we show $\mathcal{GF}_{\mathcal{B}} \cap \class{V} = \class{F}$, where $\class{F}$ is the class of all flat modules. Below we do this by adapting the argument from~\cite[Proposition~5.2]{estrada-gillespie-coherent-schemes}.

From the above comments we have $\class{F} \subseteq \mathcal{GF}_{\mathcal{B}} \cap \class{V}$, so we focus on showing the reverse containment $\mathcal{GF}_{\mathcal{B}} \cap \class{V} \subseteq \class{F}$.  So let $M \in \mathcal{GF}_{\mathcal{B}} \cap \class{V}$, and write it as $M = Z_0F$ where $F$ is an exact $\class{B}^{\otimes}$-acyclic complex of flat modules.  From~\cite[Cor.~6.4]{bravo-gillespie-hovey} or~\cite[Them.~4.2(1)/Prop.~1.7]{stovicek-deconstructible} we have a complete cotorsion pair $(\dwclass{P}, \rightperp{(\dwclass{P})})$, where $\dwclass{P}$ is the class of all complexes of projectives. So we may write a short exact sequence $$0 \xrightarrow{} F  \xrightarrow{} W  \xrightarrow{} P  \xrightarrow{} 0 $$ with $W \in  \rightperp{(\dwclass{P})}$ and $P \in \dwclass{P}$. But then using Neeman's result from~\cite{neeman-flat} (a statement in the notation we are using is also given in~\cite[Lemma~4.3]{estrada-gillespie-coherent-schemes}), one easily argues that $W \in \tilclass{F}$, the class of all exact complexes with all cycle modules flat. Since $F$ and $W$ are each exact, we see that $P$ is exact too. Moreover,  the short exact sequence is split in each degree, so tensoring with any $B \in \class{B}$, yields another short exact sequence. So since  $F$ and $W$ are each exact and $\class{B}^{\otimes}$-acyclic complexes, it follows that $P$ is an exact $\class{B}^{\otimes}$-acyclic complex too. Therefore $Z_0P$ is a projectively coresolved Gorenstein $\class{B}$-flat module. Note that by the snake lemma we get a short exact sequence $0 \xrightarrow{} Z_0F  \xrightarrow{} Z_0W  \xrightarrow{} Z_0P  \xrightarrow{} 0$. By the hypothesis, $M = Z_0F \in \class{V}$, and so we conclude  that this sequence splits. Since $Z_0W$ is flat, so is the direct summand $Z_0F$, proving $\mathcal{GF}_{\mathcal{B}} \cap \class{V} \subseteq \class{F}$.

It remains to see that the Gorenstein $\class{B}$-flat modules are closed under extensions. The reader can verify that  {\v{S}}aroch and {\v{S}}\v{t}ov{\'{\i}}{\v{c}}ek's characterizations of Gorenstein flat modules given in~\cite[Theorem~4.11]{saroch-stovicek-G-flat} generalize to any class $\class{B}$ containing the injectives and such that $(\mathcal{PGF}_{\mathcal{B}}, \mathcal{V})$ is a complete cotorsion pair. See also~\cite[Theorem~2.14]{estrada-iacob-perez-G-flat}; the proof of Estrada-Iacob-P\'erez also illustrates that the characterizations hold whenever $\class{B}$ contains the injectives and $(\mathcal{PGF}_{\mathcal{B}}, \mathcal{V})$ is a complete cotorsion pair. We state these characterizations in a Remark below. One of the characterizations that carry over is that a module $M$ is Gorenstein $\class{B}$-flat if and only if it is in the class $\leftperp{(\class{C} \cap \class{V})}$, where $\class{C}$ is the class of cotorsion modules. This class is closed under extensions, so Theorem~\ref{thm-Gor-flat-mod} applies. 
\end{proof}

Here is the promised Remark concerning~\cite[Theorem~4.11]{saroch-stovicek-G-flat}.
\begin{remark}
In addition to our blanket assumption that $\class{B}$ contains all injectives, suppose we know $(\mathcal{PGF}_{\mathcal{B}}, \mathcal{V})$ is a complete cotorsion pair. 
Then the following conditions are equivalent for an $R$-module $M$.
\begin{enumerate}
\item$M$ is Gorenstein $\mathcal{B}$-flat.
\item There is a short exact sequence of modules
\[
0 \to  F \to L \to M \to 0
\]
with $F$ flat and $L \in \mathcal{PGF}_{\mathcal{B}}$, which is also $\Hom_R(-,\mathcal{C})$-acyclic, where $\mathcal{C}$ is the class of cotorsion modules.
\item $\Ext^1_R(M,C) = 0$ for every $C \in \mathcal{C} \cap \mathcal{V}$. That is, $M \in \leftperp{(\mathcal{C} \cap \mathcal{V})}$. 
\item There is a short exact sequence of modules
\[
0 \to M \to F \to L \to 0
\]
with $F$ flat and $L \in \mathcal{PGF}_{\mathcal{B}}$.
\end{enumerate}
\end{remark}

\begin{proposition}[Analog of Proposition~\ref{prop-injective model on complexes}]\label{prop-proj-coresolved model on complexes}
Let $\class{B}$ be any class of $R^\circ$-modules for which there exists a set (not just a class) $\class{S} \subseteq \class{B}$ such that each $B \in \class{B}$ is a transfinite extension of modules in $\class{S}$. 
Then there is a cofibrantly generated projective abelian model structure on the category of chain complexes whose cofibrant objects are the 
exact $\class{B}^{\otimes}$-acyclic complexes of projectives. We call this the \textbf{exact $\boldsymbol{\class{B}^{\otimes}}$-acyclic projective model structure}.
\end{proposition}

\begin{proof}
This follows from~\cite[Theorem~6.1]{bravo-gillespie-hovey}.
One can check that a complex $P$ of projective modules is exact and $\class{B}^{\otimes}$-acyclic if and only if it is exact upon tensoring with $R\oplus B$, where $B$ is the single ``test module'' $B = \bigoplus_{N \in S} N$. Therefore, we get from~\cite[Theorem~6.1]{bravo-gillespie-hovey}, a cofibrantly generated abelian model structure on $\ch$, where the cofibrant objects are the exact $\class{B}^{\otimes}$-acyclic complexes of projectives. 
\end{proof}

\begin{theorem}[Analog of Theorem~\ref{thm-Gor-module}]\label{theorem-proj-coresolved-B-flat}
Let $\class{B}$ be a class of $R^\circ$-modules containing the injectives.  Assume there exists a set (so again, not just a class) $\class{S} \subseteq \class{B}$ such that each $B \in \class{B}$ is a transfinite extension of modules in $\class{S}$. 
Then there is a cofibrantly generated projective abelian model structure on $R$-Mod, the \textbf{projectively coresolved Gorenstein $\class{B}$-flat model structure}, whose cofibrant objects are the projectively coresolved Gorenstein $\class{B}$-flat modules
In particular, $(\mathcal{PGF}_{\mathcal{B}}, \mathcal{V})$ is a complete hereditary cotorsion pair, cogenerated by a set. 

Moreover, the Gorenstein $\class{B}$-flat model structure of Theorem~\ref{thm-Gor-flat-mod} exists and shares the same class $\class{V}$ of trivial objects as the projective model structure.

Finally, the sphere functor $S^0(-) : R\text{-Mod} \xrightarrow{} \ch$ is a right Quillen equivalence from the Gorenstein $\class{B}$-flat (resp. projectively coresolved) model structure to the exact $\class{B}^{\otimes}$-acyclic flat (resp. projective) model structure.
\end{theorem}

\begin{proof}
By Proposition~\ref{prop-proj-coresolved-B-flat} we only need to show every module $M$ has a special $\mathcal{PGF}_{\mathcal{B}}$-precover.
But by Proposition~\ref{prop-proj-coresolved model on complexes} we have the exact $\class{B}^{\otimes}$-acyclic projective model structure on chain complexes.
So for any object $M$, we can find a short exact sequence
\[
0 \xrightarrow{} X \xrightarrow{} Q \xrightarrow{}S^0(M) \xrightarrow{} 0
\]
where $Q$ is an exact $\class{B}^{\otimes}$-acyclic complex of projectives and $X$ is
trivial in the exact $\class{B}^{\otimes}$-acyclic projective model structure.  By the snake lemma, we get a short exact sequence
\[
0 \xrightarrow{} X_0/B_0X \xrightarrow{} Q_0/B_0Q \xrightarrow{} M \xrightarrow{} 0
\]
and $Q_0/B_0Q \cong Z_{-1}Q$ is projectively coresolved Gorenstein $\class{B}$-flat, by definition. 
So our goal is to show $X_0/B_0X \in \class{V}$. 
%For every exact $\class{B}^{\otimes}$-acyclic complex of projectives $Q$, and each module $N$, we have an isomorphism \begin{equation}\label{equation-sphere}\tag{$*$} \Ext^{1}_{\ch}(Q,S^{0}(N)) \cong \Ext^{1}_{R}(Z_{-1}Q, N), \end{equation} for example, see~\cite[Lemma~4.2]{gillespie-degreewise-model-strucs}.
It follows from Lemma~\ref{lemma-proj-perp} that $X_0/B_0X \in \class{V}$ if and only if $S^0(X_0/B_0X)$ is trivial in the exact $\class{B}^{\otimes}$-acyclic projective model structure. So the plan is to show below that $S^0(X_0/B_0X)$ is trivial.

But first we note that any bounded above complex of projective modules is trivial in the exact $\class{B}^{\otimes}$-acyclic projective model structure, and, any bounded below exact complex is also trivial. Indeed for any projective module $P$, we deduce that $S^n(P)$ is trivial from {\v{S}}aroch and {\v{S}}\v{t}ov{\'{\i}}{\v{c}}ek's~\cite[Theorem~4.4]{saroch-stovicek-G-flat} (see Corollary~\ref{cor-ss}) combined with the above Lemma~\ref{lemma-proj-perp}. %equation~\eqref{equation-sphere}.  
It follows that any bounded above complex of projective modules must also be trivial; for example, see~\cite[Lemma~2.3]{gillespie-AC-proj-complexes}. On the other hand, one easily verifies that for any module $N$, the disk complex $D^n(N)$ is also trivial. So~\cite[Lemma~2.3]{gillespie-AC-proj-complexes} also tells us that any bounded below exact complex is trivial.

With these observations we will argue that $S^0(X_0/B_0X)$ is trivial.
Indeed one can see that the complex $X$ has a subcomplex $A \subseteq X$, where $A$ is the shown bounded below exact complex: $\cdots \xrightarrow{} X_2 \xrightarrow{} X_1 \xrightarrow{} B_0X \xrightarrow{} 0$. As noted above, this complex is trivial, and since $X$ is trivial the quotient $X/A$ is trivial too. We note that this quotient is the complex $0 \xrightarrow{} X_{0}/B_{0}X \xrightarrow{} X_{-1} \xrightarrow{} X_{-2} \xrightarrow{} \cdots$, which in turn has another obvious subcomplex  $0 \xrightarrow{} 0 \xrightarrow{} X_{-1} \xrightarrow{} X_{-2} \xrightarrow{} \cdots$. This is a bounded above complex of projective modules and therefore it too is  trivial. This in turn implies the corresponding quotient complex, which is $S^0(X_{0}/B_{0}X)$, is trivial. This completes the proof that the short exact sequence
\[
0 \xrightarrow{} X_0/B_0X \xrightarrow{} Q_0/B_0Q \xrightarrow{} M \xrightarrow{} 0
\]
is a special $\mathcal{PGF}_{\mathcal{B}}$-precover of $M$, and gives  us the projective model structure corresponding to the Hovey triple $(\mathcal{PGF}_{\mathcal{B}}, \mathcal{V}, All)$. The construction from~\cite[Theorem~6.1]{bravo-gillespie-hovey} shows that the class of all  exact $\class{B}^{\otimes}$-acyclic complexes of projectives is filtered by a set of such complexes. The filtrations descend to a filtration on the cycles and it follows that $(\mathcal{PGF}_{\mathcal{B}}, \mathcal{V})$ is cogenerated by a set. This in turn translates to a cofibrantly generated model structure by~\cite[Section~6]{hovey}.

Again, the functor $S^0(-) : R\text{-Mod} \xrightarrow{} \ch$ is right adjoint to the functor $X \mapsto X_0/B_0X$. By~\cite[Theorem~4.2]{estrada-gillespie-coherent-schemes} we have the exact $\class{B}^{\otimes}$-acyclic flat model structure on chain complexes. We can adapt the proof of~\cite[Prop.~5.5]{estrada-gillespie-coherent-schemes} to show that these functors provide a Quillen equivalence between the flat model structures. (The proof for the projective model structures is similar. In fact, the proof for the flat case is more difficult and the proof in~\cite[Prop.~5.5]{estrada-gillespie-coherent-schemes} \emph{relies} on the existence of projective models.) Indeed the argument there shows that $X \mapsto X_0/B_0X$ preserves cofibrations and trivial cofibrations, making it a left Quillen functor. 
To show that $X \mapsto X_0/B_0X$ is a Quillen equivalence in the flat case boils down to showing the following: (i)  If $X \xrightarrow{f} Y$ is a chain map between two exact $\class{B}^{\otimes}$-acyclic complexes of flats for which the induced map $X_0/B_0X \xrightarrow{\bar{f}} Y_0/B_0Y$ is a weak equivalence, then $f$ itself must be a weak equivalence. (ii) For all cotorsion modules $C$, and short exact sequences
$0 \xrightarrow{} X \xrightarrow{} F \xrightarrow{} S^0C \xrightarrow{} 0$ with $F$ in the class ${}_{\class{B}}\tilclass{F}$ of all exact $\class{B}^{\otimes}$-acyclic complexes of flats, and $X  \in \rightperp{{}_{\class{B}}\tilclass{F}}$, then the induced short exact sequence
 $0 \xrightarrow{} X_0/B_0X \xrightarrow{} F_0/B_0F \xrightarrow{} C \xrightarrow{} 0$ must have $X_0/B_0X \in \class{V}$. 
Note that what is required to be shown for (ii) is exactly the same type of argument we did above where we showed that each module $M$ has a  special $\mathcal{PGF}_{\mathcal{B}}$-precover. In fact, the argument above will work, even with $C$ not assumed to be cotorsion, by again using   
{\v{S}}aroch and {\v{S}}\v{t}ov{\'{\i}}{\v{c}}ek's nontrivial fact from~\cite[Theorem~4.4]{saroch-stovicek-G-flat} (see Corollary~\ref{cor-ss}) that $\class{V}$ contains all flat modules. (The projective and flat models share the same class of trivial objects and each sphere complex $S^n(F)$ is trivial whenever $F$ is flat, by their result.)
To prove the above statement (i), the proof given in~\cite[Proposition~5.5]{estrada-gillespie-coherent-schemes} will readily adapt and yet again uses this fact that the thick class $\class{V}$ of trivial objects contains all flat modules. 
\end{proof}

Note that Theorems~\ref{thm-Gor-module} and~\ref{theorem-proj-coresolved-B-flat} combine to prove Theorem~\ref{them-models} from the Introduction.

\section{Gorenstein modules relative to a complete duality pair}\label{sec-D-modules}

In this section we let $\mathfrak{D} = (\class{L},\class{A})$ denote a semi-complete duality pair with $R$-modules in the projective class $\class{L}$ and $R^\circ$-modules in the injective class $\class{A}$.

\begin{corollary}\label{corollary-models}
The following abelian model structures are induced by $\mathfrak{D} = (\class{L},\class{A})$.
\begin{enumerate}
\item The \textbf{Gorenstein $\mathfrak{D}$-injective model structure} exists on $R^\circ$-Mod. It is a cofibrantly generated injective abelian model structure whose fibrant objects are the Gorenstein $\class{A}$-injective $R^\circ$-modules.
\item The \textbf{Gorenstein $\mathfrak{D}$-projective model structure} exists on $R$-Mod. It is a cofibrantly generated projective abelian model structure whose cofibrant objects are the Gorenstein $\class{L}$-projective $R$-modules, equivalently, the projectively coresolved Gorenstein $\class{A}$-flat $R$-modules.  
\item The \textbf{Gorenstein $\mathfrak{D}$-flat model structure}  exists on $R$-Mod. It is a cofibrantly generated abelian model structure whose cofibrant objects (resp. trivially cofibrant objects) are the Gorenstein $\class{A}$-flat modules (resp. flat modules). Moreover, the trivial objects in this model structure coincide with those in the Gorenstein $\mathfrak{D}$-projective model structure.
\end{enumerate}
\end{corollary}

\begin{remark}
Each model structure is Quillen equivalent to a model structure on chain complexes as described in Theorems~\ref{thm-Gor-module} and~\ref{prop-proj-coresolved model on complexes}.
\end{remark}

\begin{proof}
Gorenstein $\class{L}$-projective $R$-modules are equivalent to projectively coresolved Gorenstein $\class{A}$-flat $R$-modules by Theorem~\ref{them-projectivecomplexes}. So considering what we have shown in Theorem~\ref{thm-Gor-module}, Theorem~\ref{theorem-projectivemodels}, Theorem~\ref{prop-proj-coresolved model on complexes}, and Theorem~\ref{thm-Gor-flat-mod}, we only need to show that the injective class  $\class{A}$ contains a set $\class{S}$ for which every module in $\class{A}$ is built up as a transfinite extension of modules in $\class{S}$. But $\class{A}$ is closed under pure submodules and pure quotients by Holm and J\o rgensen's Theorem~\ref{them-duality pair purity}.
It follows from a standard argument that there exists a set $\class{S}$ as desired. For example, see~\cite[Prop.~2.8]{bravo-gillespie-hovey}.
\end{proof}

We noted in Theorem~\ref{thm-Gor-module} that $(\mathcal{W}, \mathcal{GI}_{\mathcal{B}})$ is always a perfect cotorsion pair. 
On the other hand, in the context of Proposition~\ref{prop-proj-coresolved-B-flat}, it follows from~\cite[Prop.~2.19]{estrada-iacob-perez-G-flat}  that $(\mathcal{GF}_{\mathcal{B}}, \mathcal{GC}_{\mathcal{B}})$ is always a perfect cotorsion pair. In particular, we get the following corollary. 

\begin{corollary}
Whenever $\mathfrak{D} = (\class{L},\class{A})$ is a semi-complete duality pair, then we have $(\mathcal{W}, \mathcal{GI}_{\mathcal{A}})$ and $(\mathcal{GF}_{\mathcal{A}}, \mathcal{GC}_{\mathcal{A}})$ are each perfect cotorsion pairs. 
\end{corollary}
 
By applying Corollary~\ref{corollary-models} to the duality pair $\mathfrak{D} = (\langle Flat\rangle, \langle Inj \rangle)$ from Proposition~\ref{prop-ding-thing} we get the following theorem.
 
 \begin{theorem}\label{them-dings}
The Ding injective cotorsion pair is a perfect cotorsion pair over any ring $R$. The Ding injectives form the class of fibrant objects of a cofibrantly generated injective abelian model structure on the category of modules over a ring. Therefore its homotopy category is a well-generated triangulated category.     
 \end{theorem}

%\begin{remark}\label{remark-G-flat-general}
%For a general ring $R$, we recover completeness of the Gorenstein flat cotorsion pair, and the projectively coresolved Gorenstein flat cotorsion pair, by applying Corollary~\ref{corollary-models} to the duality pair $\mathfrak{D} = (\langle Flat\rangle, \langle Inj \rangle)$ from Proposition~\ref{prop-ding-thing}.
%\end{remark}

In fact we have proved the following. 

\begin{corollary}\label{cor-n-duality}
Let $R$ be any ring and $n \geq 1$ be a natural number. We have the following special cases of interest, where all classes of modules mentioned are parts of complete cotorsion pairs, and the injective and flat pairs are perfect cotorsion pairs. 
\begin{enumerate}
\item Set $\mathfrak{D}_{\infty} :=(\class{L},\class{A})$ to be the level-absolutely clean duality pair of Example~\ref{example-level}. Then the classes of modules in Corollary~\ref{corollary-models} correspond to the Gorenstein AC-injectives, the Gorenstein AC-projectives, and the Gorenstein AC-flats.

\item For $n \geq 2$, set $\mathfrak{D}_n :=(\class{FP}_n\text{-Flat},\class{FP}_n\text{-Inj})$ to be the Bravo and P\'erez duality pairs of Example~\ref{example-BP}. Then the classes of modules in Corollary~\ref{corollary-models} correspond to what we call Gorenstein $\class{FP}_n$-injective, Gorenstein $\class{FP}_n$-projective, and Gorenstein $\class{FP}_n$-flat modules. 

\item For $n = 1$, set $\mathfrak{D}_1 := (\langle Flat\rangle, \langle Inj \rangle)$ to be the duality pair of Example~\ref{ques-G-flat}.
Then the classes of modules in Corollary~\ref{corollary-models} correspond to the Ding injectives, the projectively coresolved Gorenstein flats, and the usual Gorenstein flats.
\end{enumerate}
\end{corollary}

%\begin{question}
%Is there an even smaller duality pair, $\mathfrak{D}_0$, that will recover the Gorenstien injectives, or maybe even the Gorenstein projectives? And, do we know yet if there are Gorenstein injective modules that are not Ding injective? 
%\end{question}

\end{document}